\documentclass[11pt]{amsart}
\usepackage{amsmath, amssymb, amscd, amsthm, amsfonts, latexsym}
\usepackage{graphicx}
\usepackage{tikz-cd}
\usepackage{hyperref}
\usepackage{mathtools}
\usepackage{mathrsfs}
\usepackage{enumerate}
\usepackage{upgreek}
\usepackage{mathabx}
\usepackage{setspace}
\usepackage{color}
\usepackage{comment}
\usepackage{tikz}
\usepackage{tikz-cd}

\theoremstyle{plain}

\theoremstyle{definition}

\newcommand{\cd}{\operatorname{cd}}
\newcommand{\hd}{\operatorname{hd}}
\newcommand{\id}{\operatorname{id}}

\newcommand{\Hom}{\operatorname{Hom}}

\newcommand{\Coind}{\operatorname{Coind}}

\title{Cohomological dimension of a Lie algebra homomorphism}
\author{Nursultan Kuanyshov}

\address{Nursultan Kuanyshov, Suleyman Demirel University, Kaskelen, Kazakhstan}
\email{nursultan.kuanyshov@sdu.edu.kz, kuanyshov.nursultan@gmail.com}


\begin{document}
\maketitle

\newtheorem{theorem}{Theorem}[section]
\newtheorem{proposition}[theorem]{Proposition}
\newtheorem{lemma}[theorem]{Lemma}
\newtheorem{corollary}[theorem]{Corollary}

\theoremstyle{definition}
\newtheorem{definition}[theorem]{Definition}
\newtheorem{example}[theorem]{Example}
\newtheorem{remark}[theorem]{Remark}

\begin{abstract}
We introduce the notions of the cohomological dimension and the homological dimension of Lie algebra homomorphisms, extending the corresponding invariants for group homomorphisms. We establish their basic properties, including their behavior under monomorphisms, epimorphisms, pullbacks, composition, and restriction to the image. We also obtain a chain homotopy characterization of the cohomological dimension and compute these invariants explicitly for homomorphisms between free and abelian Lie algebras.
\end{abstract}

\section{Introduction}

The cohomological dimension of a group is one of the fundamental numerical invariants in homological algebra and geometric group theory. Introduced by Eilenberg and Ganea and further developed by Stallings \cite{St}, Swan \cite{Sw}, Bieri \cite{Bi}, and Brown \cite{Br}, it measures the largest degree in which the cohomology of a group may be nontrivial. The theory of cohomological dimension has found numerous applications in topology, geometry, and algebra, and comprehensive accounts can be found in the monographs of Bieri \cite{Bi}, Weibel \cite{We} and Brown \cite{Br}.

Recently, Mark Grant \cite{Gr} introduced the notion of the cohomological dimension of a group homomorphism, extending the classical invariant from groups to morphisms. The works of Dranishnikov and his students \cite{DD,DK,Sc} established several analogues of classical results in cohomological dimension theory and demonstrated that the resulting invariant captures important homological information about group homomorphisms. This development has motivated further investigations of homotopy and homological invariants associated with morphisms of groups; see, for example, \cite{DK,Ku1,Ku2,Ku3}.

The purpose of the present paper is to develop an analogous theory for Lie algebra homomorphisms. Since every homomorphism of Lie algebras induces natural morphisms in Chevalley--Eilenberg homology and cohomology, it is natural to ask whether numerical invariants analogous to those for group homomorphisms can be defined and whether they satisfy similar structural properties. 

In this paper we introduce the cohomological dimension $\cd(\phi)$ and the homological dimension $\hd(\phi)$ of a Lie algebra homomorphism
$\phi:\mathfrak g\longrightarrow\mathfrak h,$ defined in terms of the induced maps in Chevalley--Eilenberg cohomology and homology, respectively. We establish their basic properties, including their behavior under composition, restriction to the image, monomorphisms, and epimorphisms. A key ingredient throughout is the Lie algebra version of Shapiro's Lemma, which enables us to prove that the cohomological dimension of an injective homomorphism coincides with the cohomological dimension of its domain and to establish a monotonicity result for pullback diagrams.

We further obtain a homological characterization of the cohomological dimension by showing that if $\cd(\phi)=k$, then the chain map induced by $\phi$ between projective resolutions is chain homotopic to one that vanishes in all degrees greater than $k$. This result is the Lie algebra analogue of the corresponding theorem of De Saha and Dranishnikov \cite{DD}. Finally, we investigate the relationship between the homological and cohomological dimensions of Lie algebra homomorphisms and prove that
$\hd(\phi)\leq \cd(\phi)$.

We hope that the invariants introduced here provide a useful framework for studying homological properties of Lie algebra homomorphisms and stimulate further interactions between Lie algebra cohomology, homological algebra, and the theory of algebraic morphisms.

The paper is organized as follows. Section~2 reviews the necessary background on Lie algebras, universal enveloping algebras, Chevalley--Eilenberg homology and cohomology, and Shapiro's Lemma. In Section~3, we introduce the cohomological and homological dimensions of Lie algebra homomorphisms and establish their elementary properties. Section~4 contains the main results concerning the cohomological dimension, including its behavior under monomorphisms, pullback diagrams, and chain homotopy characterizations. In addition, we investigate the relationship between the homological and cohomological dimensions. Finally, Section~5 provides examples illustrating the theory. In particular, we compute the cohomological dimension for homomorphisms between free Lie algebras and abelian Lie algebras, showing how the introduced invariants can be explicitly determined in important classes of examples.

\section{Preliminaries}

Throughout the paper, all Lie algebras are defined over a field $K$. We assume familiarity with basic facts on Lie algebras, universal enveloping algebras, and Lie algebra (co)homology (see \cite[Chapter I]{Bo}, \cite[Chapter 7]{We}).

\subsection{Universal enveloping algebras}

For a Lie algebra $\mathfrak g$, let $U(\mathfrak g)$ denote its universal enveloping algebra. Every left $\mathfrak g$-module extends uniquely to a left $U(\mathfrak g)$-module, and conversely every left $U(\mathfrak g)$-module determines a $\mathfrak g$-module. Thus the categories of $\mathfrak g$-modules and left $U(\mathfrak g)$-modules are naturally equivalent.

Every Lie algebra homomorphism $\phi:\mathfrak g\longrightarrow\mathfrak h$ induces an algebra homomorphism
$U(\phi):U(\mathfrak g)\longrightarrow U(\mathfrak h),$
which gives rise to the corresponding restriction, induction, and coinduction functors between module categories.

\subsection{Chevalley--Eilenberg (Co)homology}

Let $\mathfrak g$ be a Lie algebra and let $M$ be a $\mathfrak g$-module. We denote by
$C_*(\mathfrak g,M)$ and $C^*(\mathfrak g,M)$ the Chevalley--Eilenberg chain and cochain complexes computing the homology and cohomology of $\mathfrak g$, respectively. Their homology groups are denoted by
$H_n(\mathfrak g,M), H^n(\mathfrak g,M).$

Recall, the (co)homological dimension of a Lie algebra $\mathfrak g$ over a field K is the maximal integer n such that $H_{n}(\mathfrak g, M)\neq 0$ ($H^{n}(\mathfrak g, M)\neq 0$) for some $\mathfrak g$-module M respectively. Analogue of Stallings and Swan theorem in Lie algebra setting is still open if the characteristics of a filed K are zero and $2$. Mikhalev, Umerbaev, and Zolotykh \cite{MUZ} constructed an example of a non-free Lie algebra of cohomological dimension one over a field of charateristic $>2$. Zusmanovich \cite{Zu} gives recent progress in this direction.

\subsection{Shapiro's Lemma}

The proofs of several of our main results rely on the Lie algebra version of Shapiro's Lemma in \cite[Shapiro's Lemma 6.3.2]{We}. 

\begin{theorem}[Shapiro's Lemma]\label{ShapiroLie}
Let $\mathfrak h\subseteq\mathfrak g$ be a Lie subalgebra and let $M$ be an $\mathfrak h$-module. Then
$$H^n\!\left(\mathfrak g,
\operatorname{Coind}_{\mathfrak h}^{\mathfrak g}M \right)\cong H^n(\mathfrak h,M)$$
naturally for every $n\ge0$, where $\operatorname{Coind}_{\mathfrak h}^{\mathfrak g}M
=\operatorname{Hom}_{U(\mathfrak h)}
\!\left(U(\mathfrak g),M
\right).$

$$H_n\!\left(\mathfrak g,U(\mathfrak g)\otimes_{U(\mathfrak h)}M \right)\cong H_n(\mathfrak h,M)$$
naturally for every $n\ge0$.
\end{theorem}

\section{Cohomological and homological dimensions of Lie algebra homomorphisms}

The functoriality of Chevalley--Eilenberg cohomology naturally leads to an analogue of the cohomological dimension of group homomorphisms in the setting of Lie algebras. Indeed, if $\phi:\mathfrak g\longrightarrow\mathfrak h$ is a Lie algebra homomorphism and $M$ is an $\mathfrak h$-module, then restriction of scalars along $\phi$ induces homomorphisms
$\phi^*:H^n(\mathfrak h,M)\longrightarrow H^n(\mathfrak g,M),$ $n\ge0.$

This motivates the following definition.
\begin{definition}
Let $\phi:\mathfrak g\longrightarrow\mathfrak h$ be a Lie algebra homomorphism. The \emph{cohomological dimension} of $\phi$ is
$$\cd(\phi)=\sup\left\{n\ge0:\exists\;\mathfrak h\text{-module }M\text{ with }
\phi^*:H^n(\mathfrak h,M)\to H^n(\mathfrak g,M)\neq0\right\}.$$
\end{definition}

The definition immediately yields the following basic properties.

\begin{proposition}\label{basiccd}
Let $\mathfrak g\stackrel{\phi}{\longrightarrow}\mathfrak h\stackrel{\psi}{\longrightarrow}\mathfrak l$
be Lie algebra homomorphisms.

\begin{enumerate}
\item $\cd(\phi)\le\min\{\cd(\mathfrak g),\cd(\mathfrak h)\}.$

\item $\cd(\id_{\mathfrak g})=\cd(\mathfrak g).$

\item $\cd(\psi\circ\phi)\le\min\{\cd(\phi),\cd(\psi)\}.$
\end{enumerate}
\end{proposition}

\begin{proof}
The first statement follows from the vanishing of
$H^n(\mathfrak g,-)$ and $H^n(\mathfrak h,-)$ above their respective cohomological dimensions. The second is immediate since the induced map on cohomology is the identity. Finally,
$
(\psi\circ\phi)^*
=
\phi^*\circ\psi^*,
$
so nontriviality of $(\psi\circ\phi)^*$ in degree $n$ implies that both $\phi^*$ and $\psi^*$ are nontrivial in degree $n$.
\end{proof}

The homological counterpart is defined similarly.

\begin{definition}
Let $\phi:\mathfrak g\longrightarrow\mathfrak h$ be a Lie algebra homomorphism. The \emph{homological dimension} of $\phi$ is
$$\hd(\phi)=\sup\left\{n\ge0:\exists\;\mathfrak h\text{-module }M\text{ with }\phi_*:H_n(\mathfrak g,M)\to H_n(\mathfrak h,M)\neq0\right\}.$$
\end{definition}

Its elementary properties are completely analogous.

\begin{proposition}\label{basichd}
Let $\mathfrak g\stackrel{\phi}{\longrightarrow}\mathfrak h\stackrel{\psi}{\longrightarrow}\mathfrak l$ be Lie algebra homomorphisms.

\begin{enumerate}
\item $\hd(\phi)\le\min\{\hd(\mathfrak g),\hd(\mathfrak h)\}.$

\item $\hd(\id_{\mathfrak g})=\hd(\mathfrak g).$

\item $\hd(\psi\circ\phi) \le \min\{\hd(\phi),\hd(\psi)\}.$
\end{enumerate}
\end{proposition}

\begin{proof}
The proof is identical to that of Proposition~\ref{basiccd}, replacing cohomology by homology.
\end{proof}

\section{Main Properties of the Cohomological Dimension}

In this section we establish the fundamental properties of the cohomological dimension of Lie algebra homomorphisms. These results are natural analogues of corresponding properties for group homomorphisms and constitute the basic structural theory of the invariant introduced in the previous section. Since the proofs in the homological setting are entirely analogous, we restrict our attention to cohomological dimension throughout this section.

\subsection{Invariance under the image}

Our first result shows that the cohomological dimension of a homomorphism depends only on its image.

\begin{theorem}\label{imageLie}
Let $\phi:\mathfrak g\longrightarrow\mathfrak h$ be a Lie algebra homomorphism, and let $\phi':\mathfrak g\longrightarrow\operatorname{Im}(\phi)=\mathfrak p$ be the induced epimorphism. Then
$$\cd(\phi)=\cd(\phi').$$
\end{theorem}

\begin{proof}
Clearly by Proposition \ref{basiccd}, we obtain that $\cd(\phi')\ge \cd(\phi),$ since
$\phi=i\circ\phi',$ where $i:\mathfrak p\hookrightarrow\mathfrak h$
is the inclusion, and hence $\phi^*=\phi'^*\circ i^*.$

It remains to prove that $\cd(\phi')\le \cd(\phi).$

Suppose $\cd(\phi')=k.$ Then there exists a $\mathfrak p$-module $M$ such that
$\phi'^*:H^k(\mathfrak p,M)\longrightarrow H^k(\mathfrak g,M)$ is nonzero.

Consider the coinduced $\mathfrak h$-module $\Coind_{\mathfrak p}^{\mathfrak h}(M)=\operatorname{Hom}_{U(\mathfrak p)}(U(\mathfrak h),M).$

Let $\alpha:\operatorname{Hom}_{U(\mathfrak p)}(U(\mathfrak h),M)\longrightarrow M $ be the canonical $\mathfrak p$-module homomorphism defined by $\alpha(f)=f(1),$ where $1$ denotes the identity element of the universal enveloping algebra $U(\mathfrak h)$.

We obtain the commutative diagram

$$
\begin{tikzcd}
H^k(\mathfrak h,\Coind_{\mathfrak p}^{\mathfrak h}M)
\arrow[r,"i^*"]
&
H^k(\mathfrak p,\Coind_{\mathfrak p}^{\mathfrak h}M)
\arrow[r,"\alpha_*"]
\arrow[d,"\phi'^*"']
&
H^k(\mathfrak p,M)
\arrow[d,"\phi'^*"]
\\
&
H^k(\mathfrak g,\Coind_{\mathfrak p}^{\mathfrak h}M)
\arrow[r,"\alpha_*"]
&
H^k(\mathfrak g,M).
\end{tikzcd}
$$

By the Lie algebra version of Shapiro's Lemma (Theorem \ref{ShapiroLie}), the composition

$$
\alpha_*\circ i^*:
H^k(\mathfrak h,\Coind_{\mathfrak p}^{\mathfrak h}M)
\longrightarrow
H^k(\mathfrak p,M)
$$

is an isomorphism. Since
$\phi'^*:H^k(\mathfrak p,M)\longrightarrow H^k(\mathfrak g,M)$ is nonzero, it follows from the commutativity of the diagram that

$$\phi^*=\phi'^*\circ i^*:H^k(\mathfrak h,\Coind_{\mathfrak p}^{\mathfrak h}M)
\longrightarrow H^k(\mathfrak g,\Coind_{\mathfrak p}^{\mathfrak h}M)$$ is also nonzero.

Hence $\cd(\phi)\ge k=\cd(\phi'),$ which proves the reverse inequality. Therefore, $\cd(\phi)=\cd(\phi').$
\end{proof}

\subsection{Injective homomorphisms}

We next determine the cohomological dimension of injective homomorphisms.

\begin{proposition}\label{prop:injective}
Let $\phi:\mathfrak g\hookrightarrow\mathfrak h$ be a monomorphism of Lie algebras. Then
$\cd(\phi)=\cd(\mathfrak g).$
\end{proposition}

\begin{proof}
By Proposition \ref{basiccd}, we already know that $\cd(\phi)\le \cd(\mathfrak g).$

Let $k=\cd(\mathfrak g).$ Then there exists a $\mathfrak g$-module $M$ such that $H^k(\mathfrak g,M)\neq 0.$

Consider the coinduced $\mathfrak h$-module $N=\Coind_{\mathfrak g}^{\mathfrak h}(M).$
By Shapiro's Lemma for Lie algebra cohomology, $H^k(\mathfrak h,N)\cong H^k(\mathfrak g,M).$
Hence
$H^k(\mathfrak h,N)\neq 0.$

Under this identification, the restriction map
$\phi^*:H^k(\mathfrak h,N)\longrightarrow H^k(\mathfrak g,N)$ corresponds to the map induced by the canonical evaluation homomorphism $N\longrightarrow M,$ which is precisely the Shapiro isomorphism (Theorem \ref{ShapiroLie}). Therefore $\phi^*$ is nonzero. It follows that $\cd(\phi)\ge k.$ Since $\cd(\phi)\le\cd(\mathfrak g)=k,$ we conclude that
$\cd(\phi)=\cd(\mathfrak g).$
\end{proof}

\subsection{Monotonicty Property}

\begin{proposition}\label{prop:pullback}
Let
$
\begin{tikzcd}
\mathfrak g' \arrow[r,"f'"] \arrow[d,hook,"i"'] &
\mathfrak h' \arrow[d,hook,"j"] \\
\mathfrak g \arrow[r,"f"'] &
\mathfrak h
\end{tikzcd}
$
be a pullback diagram of Lie algebras, where $i$ and $j$ are
monomorphisms and $\mathfrak g'=f^{-1}(\mathfrak h').$
Then
$\cd(f')\le \cd(f).$
\end{proposition}

\begin{proof}
Suppose $\cd(f')=k.$ Then there exists a $\mathfrak h'$-module $M$ such that
$f'^*:H^k(\mathfrak h',M)\longrightarrow H^k(\mathfrak g',M)$ is nonzero.

Let $N=\Coind_{\mathfrak h'}^{\mathfrak h}(M)
=\operatorname{Hom}_{U(\mathfrak h')}(U(\mathfrak h),M).$
By Shapiro's Lemma (Theorem \ref{ShapiroLie}), the composition
$$H^k(\mathfrak h,N)\stackrel{j^*}{\longrightarrow}H^k(\mathfrak h',N)\stackrel{\alpha_*}{\longrightarrow}H^k(\mathfrak h',M)$$
is an isomorphism, where
$\alpha(f)=f(1).$

Since the square commutes, $f\circ i=j\circ f',$ we obtain the commutative diagram

$$
\begin{tikzcd}
H^k(\mathfrak h,N)
\arrow[r,"j^*"]
\arrow[d,"f^*"']
&
H^k(\mathfrak h',N)
\arrow[r,"\alpha_*"]
&
H^k(\mathfrak h',M)
\arrow[d,"f'^*"]
\\
H^k(\mathfrak g,N)
\arrow[rr,"\beta_*"']
&&
H^k(\mathfrak g',M),
\end{tikzcd}
$$
where
$\beta_*=\alpha_*\circ i^*:H^k(\mathfrak g,N)\longrightarrow H^k(\mathfrak g',M).$

The map $\beta_*$ is well defined since $N$, viewed as a $\mathfrak g$-module via $f$, restricts to a $\mathfrak g'$-module via $i$.

Since the top composition
$\alpha_*\circ j^*$ is an isomorphism and
$f'^*:H^k(\mathfrak h',M)\longrightarrow H^k(\mathfrak g',M)$ is nonzero, commutativity implies that 
$\beta_*\circ f^*:H^k(\mathfrak h,N)\longrightarrow H^k(\mathfrak g',M)$
is nonzero. Hence $f^*:H^k(\mathfrak h,N)\longrightarrow H^k(\mathfrak g,N)$
must itself be nonzero.

Therefore $\cd(f)\ge k=\cd(f'),$ proving the proposition.
\end{proof}

\subsection{Homomorphisms of cohomological dimension zero}

We conclude this section with a characterization of surjective homomorphisms of cohomological dimension zero.

\begin{theorem}\label{thm:cd0}
Let $\phi:\mathfrak g\twoheadrightarrow\mathfrak h$ be a surjective homomorphism of Lie algebras. Then the following are equivalent:
\begin{enumerate}
\item $\cd(\phi)=0$;
\item $\mathfrak h=0$;
\item $\phi$ is the zero homomorphism.
\end{enumerate}
\end{theorem}

\begin{proof}
Since $\phi$ is surjective, the Theorem \ref{imageLie} gives $\cd(\phi)=\cd(\mathfrak h).$
Therefore,
$$\cd(\phi)=0\iff \cd(\mathfrak h)=0.$$

We claim that $\cd(\mathfrak h)=0\iff \mathfrak h=0.$

Indeed, if $\mathfrak h=0$, then $H^n(0,M)=0$ $(n>0)$ for every $0$-module $M$, and hence $\cd(\mathfrak h)=0$.

Conversely, if $\mathfrak h\neq0$, then the cohomological dimension of
$\mathfrak h$ is positive. In particular, there exists an $\mathfrak h$-module $M$ and an integer $n>0$ such that $H^n(\mathfrak h,M)\neq0.$ Thus $\cd(\mathfrak h)>0$.

Hence, $\cd(\phi)=0\iff\mathfrak h=0.$ Finally, since $\phi$ is surjective,
$\mathfrak h=0\iff\phi$ is the zero homomorphism. Therefore the three statements are equivalent.
\end{proof}

\subsection{Chain homotopy characterization}

The following theorem can be seen analogue of Dranishnikov and De Saha Theorem \cite[Theorem 3.3]{DD} in the case of Lie algebra homomorphisms. 

\begin{theorem}\label{chainhomotopy}
Let $\phi:\mathfrak g\to\mathfrak h$ be a Lie algebra homomorphism with
$\cd(\phi)=n$, and let
$\phi_*:(P_*(\mathfrak g),\partial_*)\longrightarrow (P_*(\mathfrak h),\partial_*')$
be the chain map between projective resolutions of the trivial
$U(\mathfrak g)$-module $k$ and the trivial $U(\mathfrak h)$-module $k$
induced by $\phi$. Then $\phi_*$ is chain homotopic to a chain map 
$\psi_*:P_*(\mathfrak g)\longrightarrow P_*(\mathfrak h)$ such that $\psi_i=0,\qquad i>n.$
\end{theorem}

\begin{proof}
For each integer $i\ge0$, let $K_i=\operatorname{Im}(\partial'_{i+1}),$ and denote by $d'_{i+1}:P_{i+1}(\mathfrak h)\twoheadrightarrow K_i$ the epimorphism obtained by restricting the boundary map
$\partial'_{i+1}$ to its image.

Consider the commutative diagram

\[
\begin{tikzcd}[column sep=large]
\cdots
\arrow[r]
&
P_{n+2}(\mathfrak g)
\arrow[r,"\partial_{n+2}"]
\arrow[d,"\phi_{n+2}"']
&
P_{n+1}(\mathfrak g)
\arrow[r,"\partial_{n+1}"]
\arrow[d,"\phi_{n+1}"']
&
P_n(\mathfrak g)
\arrow[r]
\arrow[d,"\phi_n"']
&
\cdots
\\
\cdots
\arrow[r]
&
P_{n+2}(\mathfrak h)
\arrow[r,"\partial'_{n+2}"]
\arrow[d,two heads,"d'_{n+2}"']
&
P_{n+1}(\mathfrak h)
\arrow[r,"\partial'_{n+1}"]
\arrow[d,two heads,"d'_{n+1}"']
&
P_n(\mathfrak h)
\arrow[r]
\arrow[d,two heads]
&
\cdots
\\
\cdots
\arrow[r]
&
K_{n+1}
\arrow[r]
&
K_n
\arrow[r]
&
K_{n-1}
\arrow[r]
&
\cdots
\end{tikzcd}
\]

Since $\partial'_{n+1}\partial'_{n+2}=0,$ the composition
$$P_{n+2}(\mathfrak h)\stackrel{\partial'_{n+2}}{\longrightarrow}
P_{n+1}(\mathfrak h)\stackrel{d'_{n+1}}{\longrightarrow}
K_n $$ is trivial. Consequently,$d'_{n+1}\in \Hom_{U(\mathfrak h)}(P_{n+1}(\mathfrak h),K_n)$ is an $(n+1)$-cocycle in the cochain complex
$\Hom_{U(\mathfrak h)}(P_*(\mathfrak h),K_n),$ and therefore determines a cohomology class $[d'_{n+1}]\in H^{n+1}(\mathfrak h,K_n).$

Since $\cd(\phi)=n$, the induced homomorphism
$$\phi^*:H^{n+1}(\mathfrak h,K_n)\longrightarrow
H^{n+1}(\mathfrak g,K_n)$$
is zero. Hence the cocycle
$d_n=d'_{n+1}\circ\phi_{n+1}:P_{n+1}(\mathfrak g)\longrightarrow K_n$
represents the zero cohomology class. Therefore there exists a
$U(\mathfrak g)$-homomorphism $h_n:P_n(\mathfrak g)
\longrightarrow K_n $ such that $h_n\partial_{n+1}=d_n.$

Since $P_n(\mathfrak g)$ is projective and $d'_{n+1}:
P_{n+1}(\mathfrak h)\twoheadrightarrow K_n$ is surjective, the lifting property of projective modules yields a homomorphism
$s_n:P_n(\mathfrak g)\longrightarrow P_{n+1}(\mathfrak h)$
satisfying
$d'_{n+1}s_n=h_n.$

It follows that
$$d'_{n+1}(\phi_{n+1}-s_n\partial_{n+1})=d'_{n+1}\phi_{n+1}
-d'_{n+1}s_n\partial_{n+1}=d_n-h_n\partial_{n+1}=0.$$

Hence $\operatorname{Im}(\phi_{n+1}-s_n\partial_{n+1})
\subseteq \ker(d'_{n+1}).$ Since $\ker(d'_{n+1})=\operatorname{Im}(\partial'_{n+2}),$
and $P_{n+1}(\mathfrak g)$ is projective, there exists a lift
$s_{n+1}:P_{n+1}(\mathfrak g)\longrightarrow P_{n+2}(\mathfrak h)$
such that
$\partial'_{n+2}s_{n+1}=\phi_{n+1}-s_n\partial_{n+1}.$

Repeating the above argument inductively, one constructs homomorphisms
$s_i:P_i(\mathfrak g)\longrightarrow P_{i+1}(\mathfrak h),$ $i\ge n,$ satisfying
$\partial'_{i+2}s_{i+1}+s_i\partial_{i+1}=\phi_{i+1}$ $(i\ge n).$
Thus the family $\{s_i\}_{i\ge n}$ forms a chain homotopy between
$\psi_i=0,$ $i>n,$ while $\psi_*$ remains chain homotopic to $\phi_*$. This completes the proof.
\end{proof}

The preceding theorem immediately yields the following comparison between the
homological and cohomological dimensions of a Lie algebra homomorphism.

\begin{proposition}\label{hdlecd}
Let $\phi:\mathfrak g\to\mathfrak h$ be a Lie algebra homomorphism. Then
$
\hd(\phi)\le\cd(\phi).
$
\end{proposition}

\begin{proof}
Let $n=\cd(\phi)$, and let
$
\phi_*:(P_*(\mathfrak g),\partial_*)
\longrightarrow
(P_*(\mathfrak h),\partial_*')
$
be a chain map between projective resolutions of the trivial modules induced by
$\phi$.

By Theorem~\ref{chainhomotopy}, the chain map $\phi_*$ is chain homotopic to a
chain map
$
\psi_*:P_*(\mathfrak g)\longrightarrow P_*(\mathfrak h)
$
such that
$
\psi_k=0,\qquad k>n.
$

Now let $A$ be any $\mathfrak h$-module. Tensoring with $A$ over the universal
enveloping algebras yields chain maps
$
1_A\otimes\phi_*,\,
1_A\otimes\psi_*:
A\otimes_{U(\mathfrak g)}P_*(\mathfrak g)
\longrightarrow
A\otimes_{U(\mathfrak h)}P_*(\mathfrak h),
$
which are chain homotopic and therefore induce the same homomorphism in
homology.

Since $\psi_k=0$ for all $k>n$, the induced homomorphism
$
\phi_*:
H_k(\mathfrak g,A)
\longrightarrow
H_k(\mathfrak h,A)
$
is zero for every $k>n$. Hence,
$
\hd(\phi)\le n=\cd(\phi),
$
as required.
\end{proof}

\section{Examples}

In this section we illustrate the cohomological dimension of Lie algebra homomorphisms by considering two fundamental classes of Lie algebras: free Lie algebras and abelian Lie algebras. The free case shows that the invariant is particularly rigid, taking only the values $0$ and $1$, while in the abelian case it admits a complete algebraic description in terms of the rank of the homomorphism. 

\begin{theorem}
Let $\phi:F(X)\to F(Y)$ be a homomorphism of free Lie algebras.
Then $$\cd(\phi)\in\{0,1\}.$$
Moreover, if $\phi$ is not trivial, then $\cd(\phi)=1$.
\end{theorem}

\begin{proof}
    Since cohomological dimension of free Lie algebras is one \cite{Zu} and 
    by Proposition \ref{basiccd}, we obtain that $\cd(\phi)\leq 1.$ Moreover, the image of Lie algebra $F(X)$ is subalgebra of $F(Y)$.  
    The Shirshov-Witt Theorem \cite{Sh} asserts that every subalgebra of a free Lie algebra (over a field) is free. Hence, we may assume $\phi$ to be an epimorphism. By Theorem \ref{thm:cd0}, $\cd(\phi)=0$ if and only if $\phi$ is the trivial Lie epimorphism. 
\end{proof}

\begin{theorem}\label{AbelianLie}
Let $\phi:\mathfrak a\longrightarrow\mathfrak b$ be a homomorphism of finite-dimensional abelian Lie algebras. Then 
$$\cd(\phi)=\operatorname{rank}(\phi)=\dim(\operatorname{Im}\phi).$$
\end{theorem}

\begin{proof}
Let $j=\dim(\operatorname{Im}\phi).$ Since $\mathfrak a$ and $\mathfrak b$ are vector spaces, there exist decompositions $\mathfrak a=\ker(\phi)\oplus V,$
and $\mathfrak b=\operatorname{Im}(\phi)\oplus W,$ where $\dim V=j$. Moreover,
$\phi=\phi_1\oplus g\oplus h,$ where $\phi_1:V\longrightarrow\operatorname{Im}(\phi)$ is an isomorphism,
$g:\ker(\phi)\longrightarrow0,$ and $h:0\longrightarrow W.$
Since $\phi_1$ is an isomorphism, $\cd(\phi_1)=j,$ while clearly
$\cd(g)=0,$ $\cd(h)=0.$
Therefore, $\cd(\phi)=j.$

For completeness we also verify this directly from
Chevalley--Eilenberg cohomology.

Since $\mathfrak a$ and $\mathfrak b$ are abelian,
$H^*(\mathfrak a,K)\cong\Lambda(\mathfrak a^*),$ $H^*(\mathfrak b,K)
\cong\Lambda(\mathfrak b^*),$ where $\Lambda(-)$ denotes the exterior algebra.

Choose bases $x_1,\ldots,x_n$ of $\mathfrak a$ and $y_1,\ldots,y_m$ of $\mathfrak b$ such that $\phi(x_i)=y_i,$ $1\le i\le j,$
and $\phi(x_i)=0,$ $i>j.$

Let $y_1^*,\ldots,y_m^*$ be the dual basis of $\mathfrak b^*$.
Then $\phi^*(y_i^*)=x_i^*,$ $1\le i\le j,$ and $\phi^*(y_i^*)=0,$ $i>j.$

Consequently, $\phi^*(y_1^*\wedge\cdots\wedge y_j^*)=x_1^*\wedge\cdots\wedge x_j^*\neq0.$

Hence $\cd(\phi)\ge j.$ On the other hand, $\Lambda^r(\operatorname{Im}\phi^*)=0$ for every $r>j,$ because
$\dim(\operatorname{Im}\phi^*)=j.$ Therefore,

$$\phi^*:H^r(\mathfrak b,K)\longrightarrow H^r(\mathfrak a,K)$$
is identically zero for every $r>j$.
Thus
$\cd(\phi)\le j.$

Combining the two inequalities yields

$$\cd(\phi)=j=\dim(\operatorname{Im}\phi)=\operatorname{rank}(\phi).$$
\end{proof}

Observe that the finite-dimensionality assumption on the Lie algebras in Theorem \ref{AbelianLie} is stronger than necessary. Indeed, it suffices to assume that the image of the homomorphism is finite-dimensional. We record this more general statement below.

\begin{theorem}\label{AbelianGeneral}
Let $\phi:\mathfrak a\longrightarrow\mathfrak b$ be a homomorphism of abelian Lie algebras over a field $K$. Suppose that $\operatorname{Im}(\phi)$ is finite-dimensional. Then
$$\cd(\phi)=\operatorname{rank}(\phi)=\dim(\operatorname{Im}\phi).$$
\end{theorem}

\begin{proof}
Let $r=\operatorname{rank}(\phi)=\dim(\operatorname{Im}\phi).$

Since $\mathfrak a$ and $\mathfrak b$ are abelian Lie algebras, their
Chevalley--Eilenberg differentials vanish. Hence
$$
H^*(\mathfrak a,K)\cong\Lambda(\mathfrak a^*),
\qquad
H^*(\mathfrak b,K)\cong\Lambda(\mathfrak b^*),
$$
where $\Lambda(-)$ denotes the exterior algebra.

The induced homomorphism $\phi^*:\mathfrak b^*\longrightarrow\mathfrak a^*$
is the dual of $\phi$. Since $\operatorname{rank}(\phi^*)=\operatorname{rank}(\phi)=r,$ its image is an $r$-dimensional subspace of $\mathfrak a^*$.

Choose linearly independent elements
$\eta_1,\ldots,\eta_r\in\mathfrak b^*$ whose images $\phi^*(\eta_1),\ldots,\phi^*(\eta_r)$ form a basis of $\operatorname{Im}(\phi^*)$.

Then $\phi^*(\eta_1\wedge\cdots\wedge\eta_r)
=\phi^*(\eta_1)\wedge\cdots\wedge\phi^*(\eta_r).$ Since the vectors
$\phi^*(\eta_1),\ldots,\phi^*(\eta_r)$ are linearly independent, their exterior product is nonzero. Therefore,
$\phi^*(\eta_1\wedge\cdots\wedge\eta_r)\neq0,$ which implies that
$\cd(\phi)\ge r.$

On the other hand, every element of $H^n(\mathfrak b,K)
=\Lambda^n(\mathfrak b^*)$ is a linear combination of decomposable tensors. Since $\operatorname{Im}(\phi^*)$ has dimension $r$, every exterior product of more than $r$ elements of $\operatorname{Im}(\phi^*)$ vanishes. Consequently,
$\Lambda^n(\phi^*):\Lambda^n(\mathfrak b^*)
\longrightarrow\Lambda^n(\mathfrak a^*)$
is the zero homomorphism for every $n>r$. Equivalently, 
$$ \phi^*:H^n(\mathfrak b,K)\longrightarrow H^n(\mathfrak a,K)$$
is zero for every $n>r$. Hence
$\cd(\phi)\le r.$

Combining the two inequalities yields
$$\cd(\phi)=r=\operatorname{rank}(\phi)=\dim(\operatorname{Im}\phi).$$
\end{proof}

\def\bibname{\vspace*{-30mm}{

\end{document}